\documentclass{amsproc}

\usepackage{amssymb}
\usepackage{graphicx}
\usepackage[percent]{overpic}
\newtheorem{theorem}{Theorem}[section]
\newtheorem{lemma}[theorem]{Lemma}
\newtheorem{proposition}[theorem]{Proposition}
\newtheorem{corollary}[theorem]{Corollary}
\newtheorem{fact}[theorem]{Fact}

\theoremstyle{definition}
\newtheorem{definition}[theorem]{Definition}

\theoremstyle{remark}
\newtheorem{remark}[theorem]{Remark}

\numberwithin{equation}{section}

\def\thm#1{Theorem~\ref{thm:#1}}

\def\fig#1{Figure~\ref{fig:#1}}
\def\prop#1{Proposition~\ref{prop:#1}}

\newcommand{\R}{\mathbb{R}}

\newcommand{\RD}{R_\Delta}
\newcommand{\setm}{\smallsetminus}

\newcommand{\lk}{{\operatorname{Lk}}}
\newcommand{\rib}{\operatorname{Rib}}

\begin{document}

\title[Folded ribbon unknots]{Ribbonlength of folded ribbon unknots in the plane.}


%    author one information
\author[Denne]{Elizabeth Denne}
\address{Elizabeth Denne, Department of Mathematics, Washington \& Lee University, Lexington VA}
\email{dennee@wlu.edu (corresponding author)}
\thanks{Denne was funded by 2014 and 2015 Summer Lenfest grants at Washington \& Lee University.}

%    author two information
\author[Kamp]{Mary Kamp}
\address{Mary Kamp, Department of Mathematics, Washington \& Lee University, Lexington~VA.}
\thanks{Kamp and Zhu were funded by Washington \& Lee's 2014 Summer Research Scholars Program.}

\author[Terry]{Rebecca Terry}
\address{Rebecca Terry, Department of Mathematics \& Statistics, Smith College, Northampton~MA.}
\curraddr{Department of Mathematics, The University of Utah, Salt Lake City~UT.}
\thanks{Terry was funded by the Center for Women in Mathematics at Smith College (funded by NSF grant DMS 0611020).}

\author[Zhu]{Xichen (Catherine) Zhu}
\address{Xichen Zhu, Department of Mathematics, Washington \& Lee University, Lexington~VA.}
%\thanks{Zhu was funded by Washington \& Lee's Summer Research Scholars Program 2014.}


\date{}

\begin{abstract}
We study Kauffman's model of folded ribbon knots: knots made of a thin strip of paper folded flat in the plane. The ribbonlength is the length to width ratio of such a ribbon, and it turns out that the way the ribbon is folded influences the ribbonlength. We give an upper bound of   $n\cot(\pi/n)$ for the ribbonlength of $n$-stick unknots. We prove that the minimum ribbonlength for a 3-stick unknot with the same type of fold at each vertex is $3\sqrt{3}$, and such a minimizer is an equilateral triangle. We end the paper with a discussion of projection stick number and ribbonlength.
\end{abstract}

\maketitle
%%%%%%%%%%%%%%%%%%%%%%%%
%%%%%%%%%%%%%%%%%%%%%%%%

\section{Introduction}

We can create a ribbon knot in $\mathbb R^3$ by taking a long, rectangular piece of paper, tying a knot in it, and connecting the two ends. We then flatten the ribbon into the plane, origami style, with folds in the ribbon appearing only at the corners.  Such a {\em folded ribbon knot} was first modeled by L. Kauffman \cite{Kauf05}. (He called them flat knotted ribbons.) We have illustrated two different folded ribbon unknots in \fig{ribbon-example}.

\begin{center}
\begin{figure}
\includegraphics{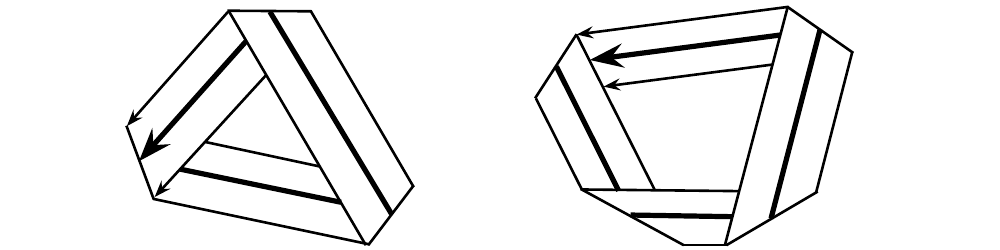}
\caption{Two different oriented folded ribbon unknots.}
\label{fig:ribbon-example}
\end{figure}
\end{center}

In 2008, B. Kennedy, T.W. Mattman, R. Raya and D. Tating~\cite{KMRT} used work of DeMaranville~\cite{DeM} to give upper bounds on the ribbonlength of various families of torus knots. (The ribbonlength is the length to width ratio of the folded ribbon knot.) They did not expect the bounds to be minimal, and in fact gave smaller versions of the $(5,2)$ and $(7,2)$ torus knots. They also gave some estimates, based on their computations, of the constants in the ribbonlength/crossing number conjecture by R. Kusner, namely that 
$$c_1\cdot Cr(K)\leq \rib(K)\leq c_2\cdot Cr(K),$$ where $c_2$, $c_2$ are unknown constants.

We pause here to note that in \cite{DSW-rib}, the first author, J.M.~Sullivan and N.~Wrinkle, have worked on the smooth analogue of folded ribbon knots. Here, the knot diagram is a smooth immersed curve in the plane.  This problem may also be thought of as a 2-dimensional analogue of the {\em ropelength problem}: that of finding the minimum amount of rope needed to tie a knot in a rope of unit diameter. (See for instance \cite{BS99,cks,gm,lsdr}.)

In this paper, we examine the folded ribbonlength of unknots. In Section 2, we review the definition of a folded ribbon knot. In Section 3, we review three different notions of ribbon equivalence, namely ribbon link, topological, and knot diagram ribbon equivalence. In Section 4, we compute the ribbonlength of $n$-stick unknots (for $n\geq 4)$ which are regular $n$-gons. This gives an upper bound on the ribbonlength of unknots with respect to link equivalence. We also give an upper bound for the ribbonlength for equilateral 3-stick unknots. This is $3\sqrt{3}$ when the each fold is of the same type, and $\sqrt{3}$ when one fold is different from the other two. In Section  5, we prove that for 3-stick unknots, where each fold is of the same type, half the ribbon width is less than or equal to the inradius of the triangular knot diagram. We then  use this idea to show that the minimum ribbonlength in this setting is indeed $3\sqrt{3}$ for an equilateral triangle. Finally, in Section 6, we discuss the $(5,2)$ torus knot example of \cite{KMRT} and show it is not ribbon link equivalent to the standard polygonal $(5,2)$ torus knot. 

%%%%%%%%%%%%%%%%%%%%%%%%%%%
%%%%%%%%%%%%%%%%%%%%%%%%%%%%%%
\section{Modeling Folded Ribbon Knots}\label{section:model}
In \cite{Kauf05}, Kauffman considered a flat ribbon immersed in the plane to be a ``bundle of parallel rays reflected by mirror segments''. He defined a flat knotted ribbon to be a choice of weaving that overlies a flat ribbon immersion. These ideas have been formalized in \cite{Rib-Smith, DSW-rib},  and we give an overview of them in this section.

%%%%%%%%%%%%%%%%
\subsection{Knot diagrams}
Recall that a {\em tame knot} is an embedding of $S^1$ in $\R^3$ (modulo reparametrizations) which is ambient isotopic to a polygonal knot. A {\em projection of a knot K} is the image of {\em K} under a projection from $\mathbb R^3$ to a plane, and a {\em knot diagram} adds gaps in a knot projection to show over and under crossing information.

\begin{definition} A {\em polygonal knot diagram} is the image of a piecewise linear (PL) immersion $K:S^1\rightarrow \R^2$, with consistent crossing information. We abuse notation and use $K$ to denote the map and its image in $\R^2$. We denote the finite number of {\it vertices} of the diagram by $v_1, ..., v_n$, and the {\it edges}~$e_i$ by $e_1=[v_1, v_2]$, \dots , $e_n=[v_n,v_1]$. If the diagram is oriented, then we assume that the labeling follows the orientation.
\end{definition}

We can formalize the notion of consistent crossing information by noting that any PL-immersion $K:S^1\rightarrow \R^2$ induces an {\em equivalence relation} $R$ on $S^1$. Namely,  $$R:=\{(x,y)\subset S^1\times S^1\, |\, K(x)=K(y)\},$$
the preimage of the diagonal $\Delta\subset \R^2\times \R^2$. We define $R_\Delta:=R\setm \Delta$ to be the set of {\em equivalent but distinct} pairs of points in $S^1\times S^1$. 

\begin{definition}\label{def:crossing}
Suppose $K:S^1\rightarrow \R^2$ is PL-immersion, then a choice
of \emph{crossing information} is a {\em continuous} function
$c\colon\RD\to\{\pm1\}$ satisfying the following two properties:
\begin{enumerate}
\item $c$ is equivariant in the sense that $c(y,x)=-c(x,y)$;
\item $c$ is transitive in the sense that $c(x,y)=+1=c(y,z)$ implies $c(x,z)=+1$.
\end{enumerate}
\end{definition}

The definitions of a polygonal knot diagram and consistent crossing information can be generalized to immersions of $S^1$ into $\R^2$, and also to maps between abstract topological spaces. This approach has been followed in \cite{DSW-rib}. We summarize the previous discussion as follows:

\begin{remark}A polygonal knot diagram $K$ in $\R^2$ requires that
\begin{enumerate}
\item  $K$ is a PL-immersion of $S^1$ in $\R^2$.
\item  For $R_\Delta$, the set of equivalent but distinct points of $K$, there is a choice of crossing information. 
\end{enumerate}
\end{remark}
Note that our definition of polygonal knot diagram did not require it to be {\em regular}. Also observe that while some polygonal knot diagrams arise from a projection of a polygonal knot to $\R^2$, many, however, do not. For example, a projection of a polygonal knot onto a plane which sends an edge to a single vertex will not result in a polygonal knot diagram. 
Conversely, it is not always the case that a polygonal knot can be built from a  polygonal knot diagram. The simplest example is the trefoil knot. It is well known (see for instance \cite{Adams, Calvo}) that the minimum stick index\footnote{Recall that the {\it stick index} of a knot K is defined to be the least number of line segments needed to construct a polygonal embedding of K.}  
for the trefoil knot is six.  

%%%%%%%%%%%%%%%%%%%%%%%%%%%%
\subsection{Ribbons of width $w$}

Given a polygonal knot diagram $K$, how might we construct a folded ribbon knot?  We take a geometric rectangle and map it to the plane by a pathwise isometry that is an immersion everywhere but at the fold lines. 

\begin{definition} Given an oriented polygonal knot diagram $K$, we define the {\em fold angle} at vertex $v_i$ to be the angle $\theta_i$ (where $0\leq\theta_i\leq \pi$) between edges $e_{i-1}$ and $e_i$.  
\end{definition}
In \fig{ribbon-construct} (left), the fold angle is $\theta_i=\angle ECF$. We say that it is positive, since $e_i$ is to the left of $e_{i-1}$. If $e_i$ were to the right of $e_{i-1}$, then it would be negative.
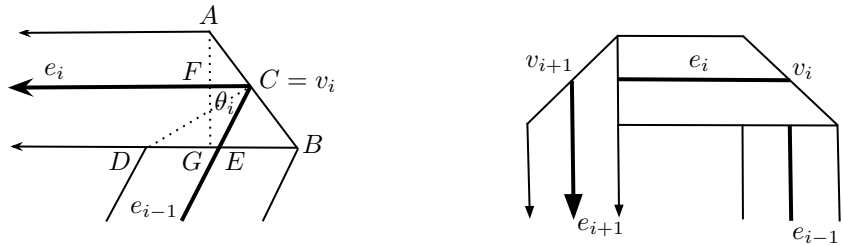
\begin{figure}[htbp]
\begin{center}
\begin{overpic}{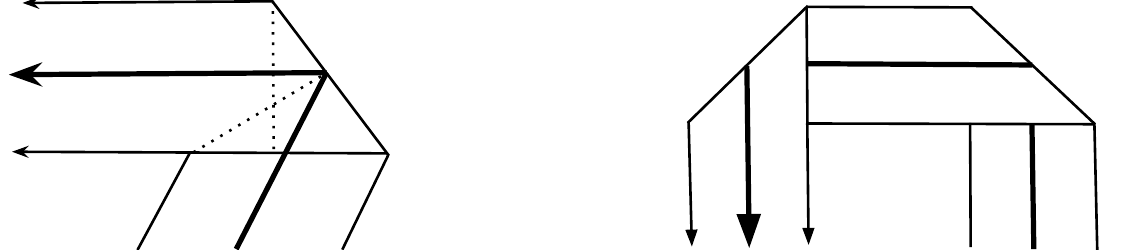}
\put(23,23){$A$}
\put(35,8){$B$}
\put(30,15.5){$C=v_i$}
\put(24.8,13){$\theta_i$}
\put(26,6){$E$}
\put(12.5,6){$D$}
\put(21,16.5){$F$}
\put(21,6){$G$}
\put(5,17){$e_i$}
\put(15,1){$e_{i-1}$}

\put(80,18){$e_i$}
\put(92,-2){$e_{i-1}$}
\put(67,-1){$e_{i+1}$}
\put(92,17){$v_i$}
\put(61,18){$v_{i+1}$}
\end{overpic}
\caption{On the left, a close-up view of a ribbon fold. On the right, the construction of the ribbon centered on edge $e_i$. }
\label{fig:ribbon-construct}
\end{center}
\end{figure}

\begin{definition}\label{def:FoldedRibbon} Given an oriented polygonal knot diagram $K$, an oriented folded ribbon knot of {\em width} $w$, denoted $K_w$, is constructed as follows:
 \begin{enumerate}

\item  First, construct the fold lines. At each vertex $v_i$ of $K$, find the fold angle~$\theta_i$.  If $\theta_i=\pi$, there is no fold line.
If $\theta_i<\pi$, place a fold line of length $w/\cos(\frac{\theta_i}{2})$ centered at $v_i$ perpendicular to the angle bisector of $\theta_i$. 
\item Second, add in the ribbon boundaries. For each edge $e_i$, join the ends of the fold lines at  $v_i$ and $v_{i+1}$. Each boundary line is parallel to, and distance $w/2$ from $K$.
\item The ribbon inherits an orientation from $K$.
\end{enumerate}
\end{definition}

This construction is illustrated in \fig{ribbon-construct}. On the left, the fold angle is $\theta_i=\angle ECF$, the angle bisector is $DC$, and the fold line is $AB$.  Using the geometry of the figure we see that $\angle GAB=\theta_i/2$ in right triangle $\triangle AGB$. Thus $|AB|=w/\cos(\frac{\theta_i}{2})$ guarantees the ribbon width $|AG|=w$. 

Observe that near a fold line, there is a choice of which ribbon lies above the other. Thus a polygonal knot diagram with $n$ vertices has $2^n$ possible folded ribbon knots depending on the choices made. 

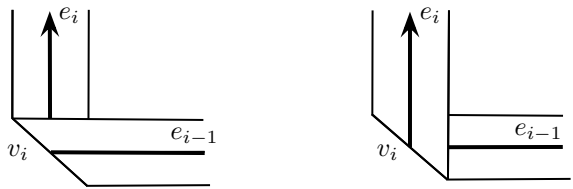
\begin{figure}[htbp]
\begin{center}
\begin{overpic}{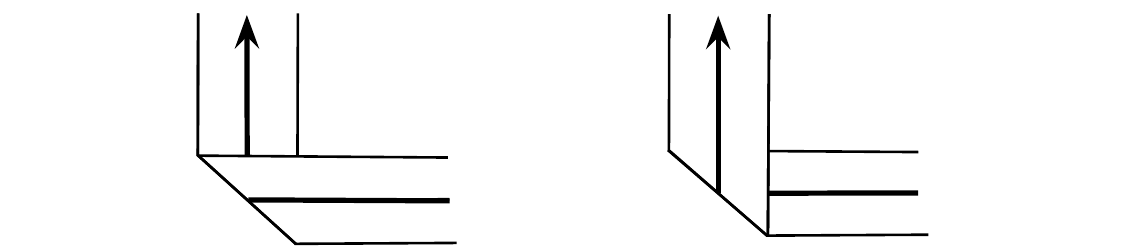}
\put(36, 6){$e_{i-1}$}
\put(23, 20){$e_{i}$}
\put(17, 4){$v_i$}

\put(76, 6.5){$e_{i-1}$}
\put(65, 20){$e_{i}$}
\put(60, 4){$v_i$}
\end{overpic}
\caption{A right underfold (left) and a right overfold (right).}
\label{fig:folding-info}
\end{center}
\end{figure}

\begin{definition}
Let $K_w$ be an oriented folded ribbon knot which is immersed (except for the fold lines). There is an {\em overfold} at vertex $v_i$ if the ribbon corresponding to segment $e_{i}$ is over the ribbon of segment $e_{i-1}$ (see \fig{folding-info} right). Similarly, there is an {\em underfold} if the ribbon corresponding to $e_{i}$ is under the ribbon of $e_{i-1}$. We also refer to {\em left} and {\em right} overfolds or underfolds depending on the sign of the fold angle. The choice of over or underfold at each vertex of $K_w$ is called the {\em folding information}, and is denoted by $F$.
\end{definition}
If we reverse the orientation of $K$, overfolds become underfolds and vice-versa.
An example showing three unknots with the same knot diagram but different folding information can found later in Figure~\ref{fig:4unknot}. 

Note that we can define the same kind of equivalence relation on ribbons $K_{w,F}$ as on knot diagrams $K$, and again define $R_\Delta$ to mean the set of distinct but equivalent points on the folded ribbon. We now combine all these ideas to give the following:

\begin{definition}\label{folded-ribbon} Given an oriented knot diagram $K$, we say the folded ribbon $K_{w,F}$ of width $w$ and folding information $F$ is {\em allowed} provided
\begin{enumerate}
\item The ribbon has no singularities (is immersed), except at the fold lines.
\item $K_w$ has a choice of crossing information, and moreover this agrees
\begin{enumerate} \item with the folding information given by $F$, and
\item with the crossing information of the knot diagram $K$.
\end{enumerate}
\end{enumerate}
\end{definition}

When a folded ribbon $K_{w,F}$ is allowed, the continuous crossing information on $\RD$ means that relative orders are fixed on path-connected subsets of $\RD$. This formalizes our intuition that a straight ribbon segment cannot ``pierce'' a fold, it either lies entirely above or below the fold, or lies between the two ribbons segments joined at the fold. %We also note that a fold (and nearby ribbon) may not be ``folded'', such a ribbon has non-allowed singularities.

 \begin{definition}
The \emph{width of a knot diagram $K$} is the widest ribbon allowed. That is, $w(K):=\sup\{w\,\vert\, \text{width $w$ allowed} \}$. Similarly, the {\em width of a knot diagram $K$ with folding information $F$} is  $w(K_F):=\sup\{w\,\vert\, \text{width $w$ allowed for $K_{w,F}$} \}$.
\end{definition}

 Observe that not all ribbon widths give an allowed folded ribbon knot. However, we can construct folded ribbon knots for ``small enough'' widths, as proved in \cite{Rib-Smith}. 
 \begin{proposition}[\cite{Rib-Smith}] Given any regular polygonal knot diagram $K$ and folding  information $F$, there is a constant $C>0$ such that an allowed folded ribbon knot $K_{w,F}$ exists for all $w<C$. 
\end{proposition}

\begin{remark} We observe that the topology of the folded ribbon knot depends on the number of edges of the knot diagram:  $K_{w,F}$ is a topological annulus when $K$ has an even number of edges, and is a topological M\"obius strip when $K$ has an odd number of edges.  \end{remark}

%%%%%%%%%%%%%%%%%%%%%%%%%%%%%%%%%%%
%%%%%%%%%%%%%%%%%%%%%%%%%%%%%%%%%%%%
\section{Ribbon Equivalence}
\label{section:linkingnumber}

When given two folded ribbon knots, one might wonder when they are the same and different. To do this, we first begin by defining ribbon linking number.

%%%%%%%%%%%%%%%%%%%%%%%%%

\subsection{Ribbon Linking number}

The {\em linking number} is an invariant from knot theory used to determine the degree to which components of a link are joined together.  Given an oriented two component link $L=A\cup B$, recall that the linking number  $\lk(A,B)$ is defined to be one half the sum of $+1$ crossings and $-1$ crossings between $A$ and $B$. 

\begin{figure}[htbp]
\begin{center}
\begin{overpic}{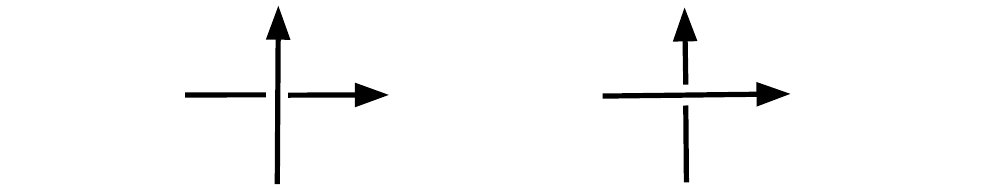}
\end{overpic}
\end{center}
\caption{The crossing on the left is labelled -1, the crossing on the right +1.}
\label{crossingvalue}
\end{figure}

Although we have described the construction of folded ribbon diagrams in $\mathbb R^2$, we can also consider the ribbons that these diagrams represent in $\mathbb R^3$. That is, as {\em framed knots}.

 \begin{definition} \label{ribbonlinkingnumber}
 Given an oriented folded ribbon knot $K_{w,F}$, we define the {\em ribbon linking number} to be the linking number between the knot diagram and one boundary component of the ribbon.  We denote this as $\lk(K_{w,F})$, or $\lk(K_w)$.
 \end{definition}
 
In \fig{ribbon-example}, the 3-stick unknot on the left  has
ribbon linking number $+1$, while the the 4-stick unknot on the right has ribbon linking number $-2$.

%%%%%%%%%%%%%%%%%%%%%%%%%%
\subsection{Ribbon Equivalence}

We are now ready to define three different kinds of ribbon equivalence. We start with the most restrictive. If we consider two ribbons in space, we say they are equivalent if there is an ambient isotopy of $\R^3$ that takes one to the other. Viewed in the plane this gives the following:

\begin{definition}(Link equivalence)\label{def:re}
Two oriented folded ribbon knots are {\em (ribbon) link equivalent} if they have equivalent knot diagrams with the same ribbon linking number.
\end{definition}

For example, the left and center folded ribbon unknots in \fig{4unknot} are link equivalent, while the one on the right is not link equivalent to them. This example shows that there can be different looking folded ribbon knots with the same ribbon linking number.

\begin{figure}[htbp]
\begin{center}
\begin{overpic}{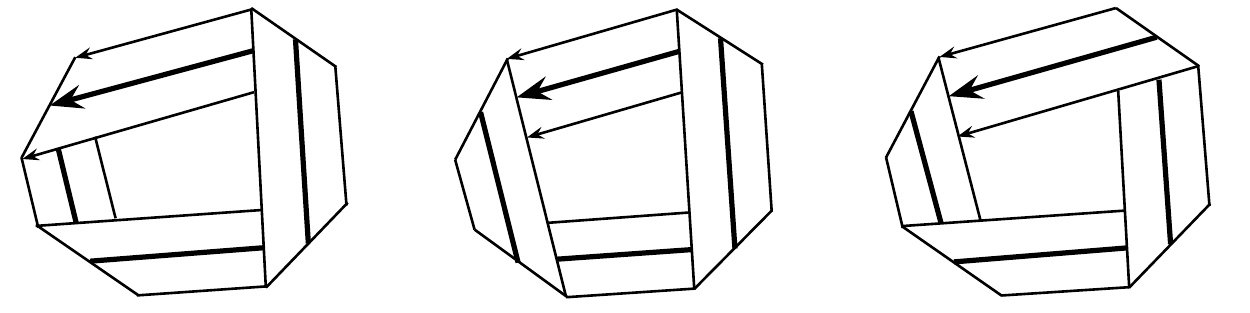}
\end{overpic}
 \caption{The left and center 4-stick folded ribbon unknots have ribbon linking number 0, while the one on the right has ribbon linking number $-4$. } 
\label{fig:4unknot}
\end{center}
\end{figure}

\begin{definition}(Topological equivalence)\label{def:te}  Two oriented  folded ribbon knots are {\em topologically (ribbon) equivalent} if they have equivalent knot diagrams and, when considered as ribbons in $\mathbb R^3$, both ribbons are topologically equivalent to a M\"obius strip or both ribbons are topologically equivalent to an annulus.  
\end{definition}

For example, all of the 4-stick folded ribbon unknots in \fig{4unknot} are topologically equivalent.

\begin{definition}(Knot diagram equivalence)\label{def:kde}  Two folded ribbon knots are {\em knot diagram equivalent} if they have equivalent knot diagrams.
\end{definition}

 For example, the 3-stick and 4-stick folded ribbon unknots in \fig{ribbon-example} are knot diagram equivalent, but are not topologically equivalent, nor link equivalent. 

%%%%%%%%%%%%%%%%%%%%%%%%%%%%%%%%%%%
%%%%%%%%%%%%%%%%%%%%%%%%%%%%%%%%%%%
\section{Ribbonlength}
\label{section:ribbonlength}

Given a particular folded ribbon knot, it is very natural to wonder what is the least length of paper needed to tie it.  More formally, we define a scale invariant quantity, called {\em ribbonlength}, as follows: 

\begin{definition}  The {\em (folded) ribbonlength}, $\rib(K_{w,F})$, of a folded ribbon knot $K_{w,F}$ is the quotient of the length of $K$ to the width $w$:
$$\rib(K_{w,F})=\frac{length(K)}{w}.$$
\end{definition}     

\begin{remark} The {\em ribbonlength problem} asks us to minimize the ribbonlength of a folded ribbon knot, while staying in a fixed topological knot type. That is, with respect to knot diagram equivalence of folded ribbon knots.  We can also ask to minimize the ribbonlength of folded ribbon knots with respect to topological and link equivalence.
\end{remark}

The ribbonlength problem remains open. Others \cite{Kauf05,KMRT} have found upper bounds on the ribbonlength of the trefoil knot, figure-8 knot and some infinite families of torus knots. There, the ribbonlength was found with respect to knot diagram equivalence. Some of these examples, like the trefoil knot, may well be ribbonlength minimizers. However, none of these examples were explicitly computed to bound ribbonlength with respect to topological or link equivalence.

\subsection{Unknots}
Any polygonal unknot diagram can be reduced to a 2-stick unknot, thus the minimum ribbonlength of any unknot (with respect to knot diagram equivalence) is~$0$. However, we do not expect this to be the case when considering topological or link equivalence.  We can easily get an upper bound on the ribbonlength of $n$-stick unknots by considering the case where the knot diagrams are regular $n$-gons. 

\begin{figure}[htbp]
\begin{center}
\begin{overpic}{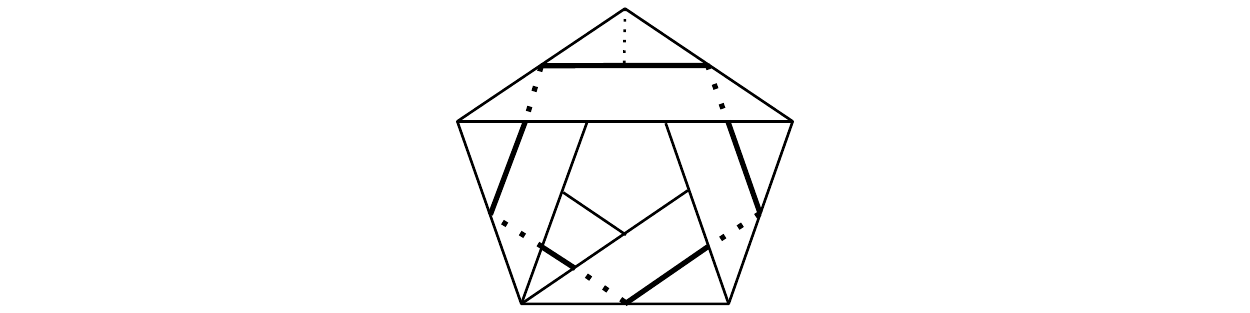}
\put(49,25){$A$}
\put(64,15){$B$}
\put(59,0){$C$}
\put(37,0){$D$}
\put(34,15){$E$}
\put(41,20){$F$}
\put(57,20){$G$}
\put(49,17){$H$}
\end{overpic}
\caption{Finding ribbonlength for a 5-stick folded ribbon unknot.}
\label{5unknot}
\end{center}
\end{figure}
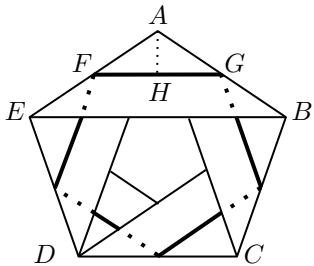

\begin{proposition} \label{prop:min-unknot} The ribbonlength of an $n$-stick folded ribbon unknot (for $n\ge4$) is less than or equal to $n \cot(\frac{\pi}{n})$.
\end{proposition} 

\begin{proof}  
Assume that the width $w=1$, and that the knot diagram is a regular $n$-gon. Shrink the diagram until the folds meet and form form a regular $n$-gon (for example the pentagon $ABCDE$ in Figure~\ref{5unknot}). For $n\geq 5$, there is a hole in the middle of the ribbon, since the ends of the folds meet on the unbounded component of $\R^2\setm K$. By construction, the vertices of the knot diagram are midpoints of the sides of the $n$-gon constructed from the folds.

Using Figure~\ref{5unknot} as a guide, drop a perpendicular $AH$ to side $FG$ of the unknot. Recall that the interior angle of a regular $n$-gon is $\pi(\frac{n-2}{n})$, hence $\angle AGH= {\pi}/{n}$ and $|HG|=\frac{1}{2}\cot(\pi/n)$. Using symmetry, we find the total length of the knot diagram to be $n\cot(\pi/n)$.  This is an upper bound on the minimum ribbonlength with respect to any type of ribbon equivalence. 
\end{proof}    

In \cite{Rib-Smith}, the first, third and other coauthors show how the minimum ribbonlength of the 5-stick folded unknot depends on the kind of ribbon equivalence.  There, we found that when minimizing ribbonlength of unknots with respect to topological equivalence, the 3-stick unknot case is important. We also gave more examples of ribbonlength computations for non-convex 4-stick unknots.

%%%%%%%%%%%%%%%%

\subsection{3-stick unknot} For the 3-stick unknot, there are two types of folding information, up to rotation and flipping the entire ribbon over. Either all of the folds are of the same type, or one is different from the other two, as illustrated in Figure~\ref{fig:ooo-oou}.

\begin{center}
\begin{figure}[htpb]
 \begin{overpic}{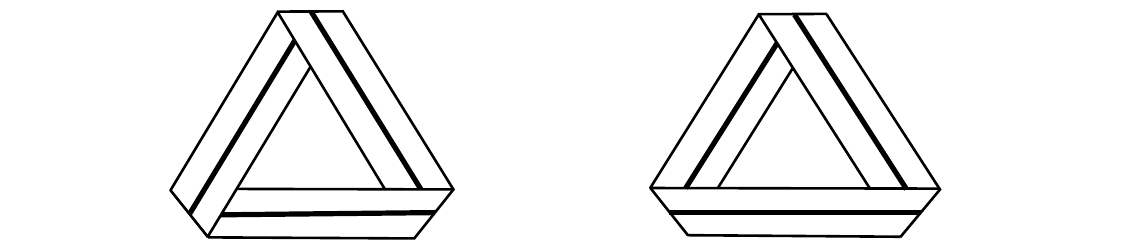}
 \put(26,23){$A$}
\put(12,2){$B$}
\put(40,2){$C$}
 \put(69,23){$A$}
\put(55,2){$B$}
\put(84,2){$C$}
\end{overpic}
\caption{On the left, all the folds are of the same type. On the right, one fold is different from the other two.}
\label{fig:ooo-oou}
\end{figure}
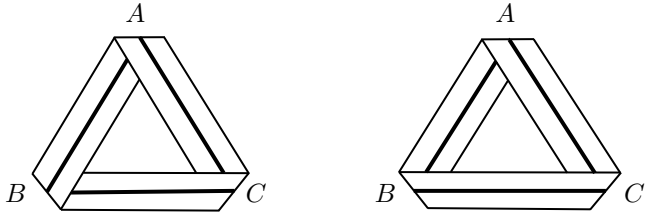
\end{center}

\begin{proposition} \label{prop:ooo-oou} The minimum ribbonlength of an 3-stick folded ribbon unknot $K_{w,F}$ is less than or equal to
\begin{enumerate}
\item $3\sqrt{3}$ \ when the folds are all the same type,
\item $\sqrt{3}$ \ when one fold is of different type to the other two.
\end{enumerate}
\end{proposition} 

\begin{proof}  Assume that the width is $w=1$, and that the knot diagram is an equilateral triangle. Also assume that the folds are all of the same type, as illustrated on the left in Figure~\ref{fig:ooo-oou}. Shrink $\triangle ABC$ until the center of the ribbons' edges meet at the incenter of the triangle. \fig{equilateral} shows the incenter $I$ of $\triangle ABC$, with $AE$ and $IC$ angle bisectors. By construction, $\angle IAD=\pi/6$, and $w/2=|ID|=1/2$. 
Thus in $\triangle IAD$, we have $ \vert AD \vert=\vert ID\vert \cot{(\frac{\pi}{6})} = \frac{\sqrt{3}}{2}$. Using symmetry, we deduce that $\vert AC \vert = 2\vert AD \vert = \sqrt{3}$, and therefore the ribbonlength is $3\sqrt{3}$.

\begin{figure}[htbp]
\begin{center}
\begin{overpic}{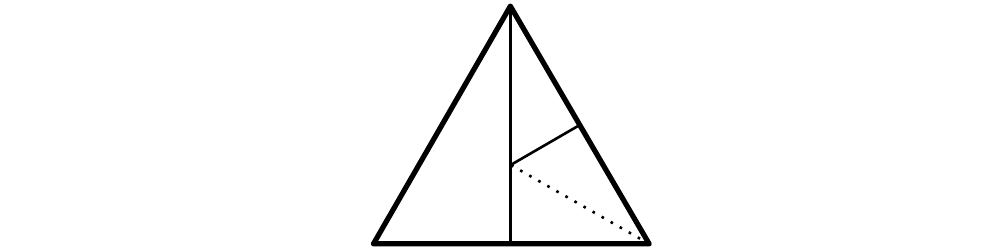}
\put(50,26){$A$}
\put(33,0){$B$}
\put(66,0){$C$}
\put(48,8){$I$}
\put(59,12){$D$}
\put(50,-3){$E$}
\end{overpic}
\caption{Equilateral triangle $\triangle ABC$, with incenter $I$.}
\label{fig:equilateral}
\end{center}
\end{figure}
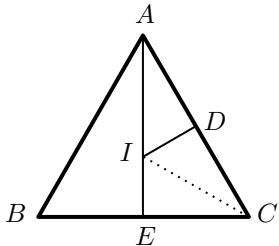

Now assume that one fold is of different type to the others, as illustrated on the right in Figure~\ref{fig:ooo-oou}. Here, we can keep shrinking $\triangle ABC$ until the edge of the ribbon along side $BC$ meets the fold at $A$. At this point, the folds form an equilateral triangle (as in the proof of \prop{min-unknot}, the ends of the folds touch in the unbounded component of $\R^2\setm K$).  Thus in \fig{equilateral}, $w/2=|AE|=1/2$, and hence $\frac{\vert AE\vert}{\vert AC\vert}=\cos (\frac{\pi}{6}) = \frac{\sqrt{3}}{2}.$
Therefore $|AC| = 1/\sqrt{3}$, and ribbonlength is $3/\sqrt{3} = \sqrt{3}$. \end{proof}

In conclusion, the ribbonlength of an unknot up to knot diagram equivalence is 0. The ribbonlength of an unknot is less than or equal to $\sqrt{3}$ when the ribbon is equivalent to a M\"obius strip, and is 0 when the ribbon is equivalent to an annulus. When ribbon linking number is taken into consideration the situation is more complex. For example, for 3-stick unknots with ribbon linking number $\pm 3$ the ribbonlength is less than or equal to $3\sqrt{3}$, and is less than or equal to $\sqrt{3}$ for 3-stick unknots with ribbon linking number $\pm 1$. In general, the ribbonlength of an $n$-stick unknot is less than or equal to $n\cot(\pi/n)$ with respect to ribbon link equivalence.

%%%%%%%%%%%%%%%%%%%%%%%%%%%%%%%%%%%%%%%%
%%%%%%%%%%%%%%%%%%%%%%%%%%%%%%%%%%%%%%%%
\section{Local structure of folded ribbon knots}\label{maxwidth}
In this section, we show the relationship between the ribbon width and inradius of a 3-stick unknot diagram, then generalize it to the local structure of more general folded ribbon knots. We then show that an equilateral 3-stick knot diagram has minimum ribbonlength when all folds are of the same type.
%%%%%%%%%%%%%%%%%%
\subsection{Width and inradius}
The key step in proving \prop{ooo-oou} was realizing the ribbon touches in the interior of the triangle at the incenter. Recall that the {\em inradius} is the distance from the incenter to any edge of the triangle, and is the radius of the largest circle inscribed in the triangle.

\begin{theorem} \label{thm:local-structure}
Given a nondegenerate 3-stick folded ribbon unknot $K_{w,F}$, where the folds are all of the same type, then $\displaystyle \frac{w}{2} \leqslant r_{in}$, where $r_{in}$ is the inradius of the knot diagram.
\end{theorem}

\begin{figure}[htbp]
\begin{center}
\begin{overpic}{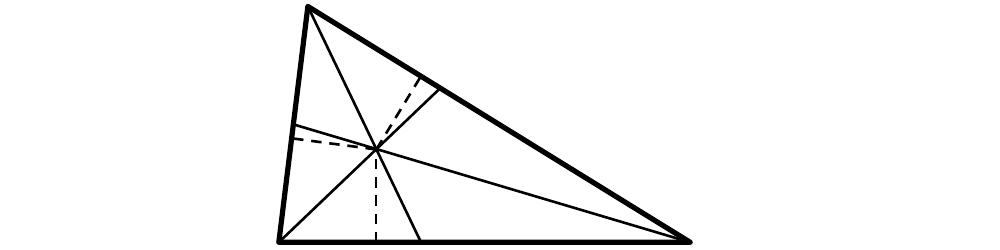}
\put(30,25){$A$}
\put(24,0){$B$}
\put(70,0){$C$}
\put(36,-2.5){$G$}
\put(25,9){$H$}
\put(41,19){$J$}
\put(26,13){$E$}
\put(44.5,16.5){$F$}
\put(42,-2.5){$D$}
\put(40,10){$I$}
\end{overpic}
\end{center}
\caption{The knot diagram is represented by triangle $\triangle ABC$ with incenter $I$.}
\label{fig:crossingproof}
\end{figure}
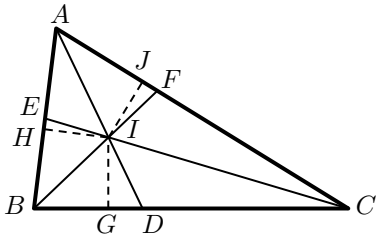

\begin{proof}
Suppose the knot diagram is represented by $\triangle ABC$ as shown in \fig{crossingproof}. There, $AD$, $BF$ and $CE$ are the angle bisectors intersecting at point $I$. Line segments $IJ$, $IG$, respectively $IH$ are perpendiculars from $I$ to sides $AC$, $BC$ and $AB$ respectively, and $\vert IG\vert  =\vert IH\vert  =\vert IJ\vert =r_{in}$, the inradius of $\triangle ABC$.

 Assume by way of contradiction that $\frac{w}{2} > r_{in}$. This means $I$ must be contained in all three ribbons adjacent to each side. Without loss of generality, assume that each fold is an overfold as we traverse $\triangle ABC$ in a counterclockwise direction. At $B$, the ribbon near  $BC$ is over the ribbon near $BA$ and at $C$, the ribbon near $CA$ is over the ribbon near $CB$. Since crossing information is continuous, at point $I$ the ribbon near $CA$ is over the ribbon near $AB$. However, at $A$, the ribbon near $AB$ is over the ribbon near $AC$, a contradiction. Therefore we have $\frac{w}{2} \leqslant r_{in}.$
\end{proof}

This result can easily be generalized to other regular $n$-gons. There, the inradius is the radius of the largest inscribed circle in the $n$-gon. Unfortunately, we do not get better results for the ribbonlength than found in \prop{min-unknot}. We previously observed that there is a hole in the center of the ribbon for $n\geq 5$, and so $w/2$ is strictly smaller than the inradius. 

However, we may generalize \thm{local-structure} to apply to the local structure of more general folded ribbon knots. We now zoom in to a portion of a knot diagram consisting only of a nondegenerate triangle. Here the vertices of the triangle are either double points or vertices of the knot diagram. We first examine the possible choices for crossing information at each vertex up to symmetry. 

\begin{enumerate}
\item The two cases for triangle with folds at all three vertices are shown in \fig{ooo-oou}.

\item The two cases where the vertices of the triangle are at three crossings are shown in \fig{3-crossings}.  %On the left, as we traverse the triangle in a clockwise direction, the crossing information is over-over-over; on the right, it is over-over-under.

\begin{figure}[htbp]
\begin{center}
\begin{overpic}{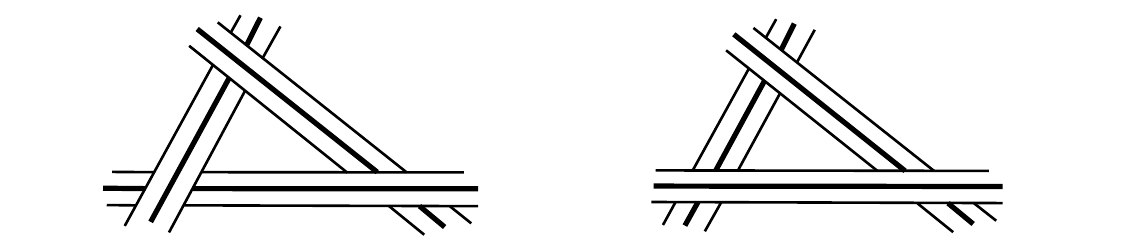}
\end{overpic}
\end{center}
\caption{Different crossing information where there are three crossings.}

\label{fig:3-crossings}
\end{figure}

\item The three cases where the vertices of the triangle are at one fold and two crossings are shown in \fig{1-fold-2-cr}. %On the left, as we traverse the triangle in a clockwise direction, the crossing information is over-over-over;  in the middle and on the right, it is over-over-under.

\begin{figure}[htbp]
\begin{center}
\begin{overpic}{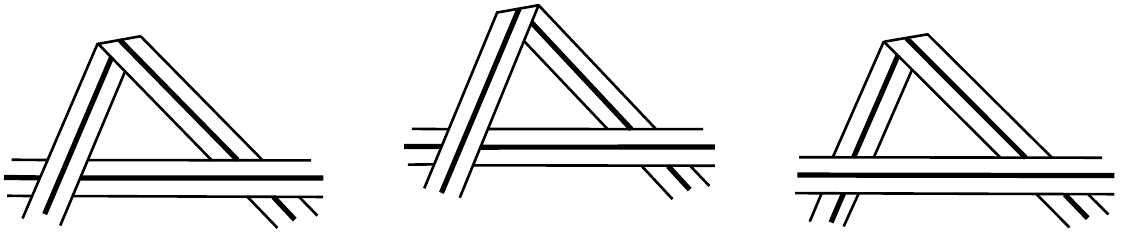}
\end{overpic}
\end{center}
\caption{Different crossing information where there are two crossings and one fold.}
\label{fig:1-fold-2-cr}
\end{figure}

\item The three cases where the vertices of the triangle are at two folds and one crossing are shown in \fig{2-folds-1-cr}. %On the left, as we traverse the triangle in a clockwise direction, the crossing information is over-over-over;  in the middle and on the right, it is over-over-under.

\begin{figure}[htbp]
\begin{center}
\begin{overpic}{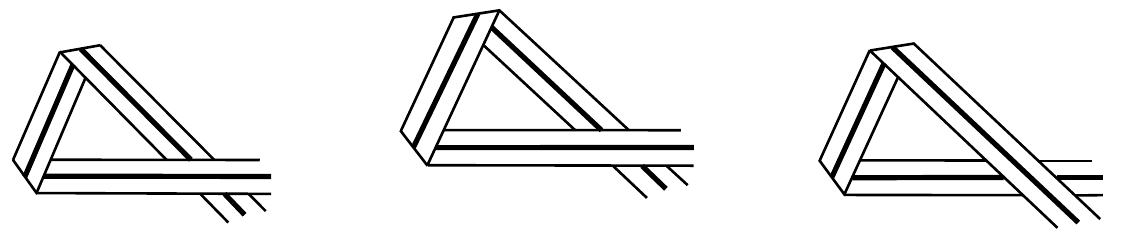}
\end{overpic}
\end{center}
\caption{Different crossing information where there are two folds and one crossing.}
\label{fig:2-folds-1-cr}
\end{figure}

\end{enumerate}

In the leftmost example in each of Figures~\ref{fig:3-crossings},~\ref{fig:1-fold-2-cr}, and~\ref{fig:2-folds-1-cr}, the folding and/or crossing information is the same. Namely, as we traverse the triangle in a clockwise direction the crossing information is over-over-over. In each case, the arguments from \thm{local-structure} carry over immediately. We thus deduce the following corollary.

\begin{corollary} Let $K_{w,F}$ be a folded ribbon knot, and suppose part of the knot diagram consists of a triangle with inradius $r_{in}$. If the folding and/or crossing information is the same as the triangle is traversed (for example, over-over-over), then $\displaystyle \frac{w}{2} \leqslant r_{in}$. \qed
\end{corollary}

\begin{remark} The astute reader will realize that the arguments in \thm{local-structure} should carry over to general polygonal polygonal regions as well. We need the notion of medial axis to make these ideas precise. This approach has been followed by the first author in \cite{DSW-med,DSW-rib} for regions of a knot diagram bounded by curves of finite total curvature (of which polygons are a subset).
\end{remark}
%%%%%%%%%%%%%%%%%%%%%%%%%%%%
%%%%%%%%%%%%%%%%%%%%%%%%%%%%
\subsection{Minimizing ribbonlength for the 3-stick unknot}

We end this section by showing an equilateral 3-stick unknot minimizes ribbonlength. The proof of this result relies on two well-known geometric properties of the inradius of a triangle (see for instance \cite{PW,Wil}). We outline these ideas below.

\begin{fact} \label{lem:areainradius}
For any triangle, the area  equals one-half the product of its inradius with its perimeter. 
\end{fact}

\begin{figure}[htbp]
\begin{center}
\begin{overpic}{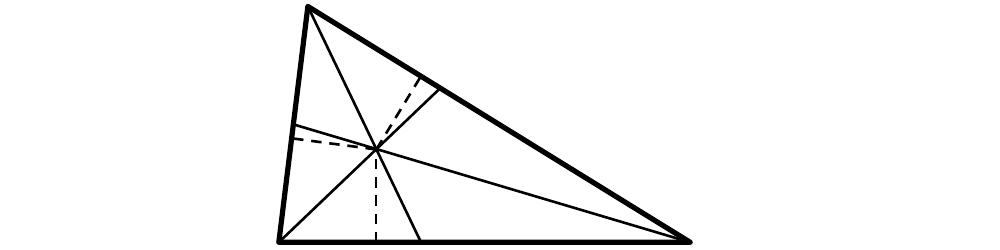}
\put(30,25){$A$}
\put(24,0){$B$}
\put(70,0){$C$}
\put(36,-2.5){$G$}
\put(25,9){$H$}
\put(41,19){$J$}
\put(26,13){$E$}
\put(44.5,16.5){$F$}
\put(42,-2.5){$D$}
\put(40,10){$I$}
\end{overpic}
\end{center}
\caption{Triangle $\triangle ABC$ with incenter $I$.}
\label{fig:inradius}
\end{figure}

In Figure~\ref{fig:inradius}, $AD$, $BF$, $CE$ are angle bisectors, and we may view triangle $\triangle ABC =\triangle AIB +\triangle BIC+\triangle CIA$. Since the inradius $r_{in}$ is the height of each of the smaller triangles,  the area of triangle $\triangle ABC$ may be written as
$$\mathcal {A}_{\triangle ABC} =\frac{1}{2} (\vert  AB \vert + \vert BC  \vert +  \vert AC  \vert) \cdot  r_{in}. $$

Now for a fixed perimeter triangle, the inradius is maximized when the area of the triangle is largest. It is straightforward to show that the area of a fixed perimeter triangle is maximized when it is an equilateral triangle (for example using Heron's formula and the arithmetic mean -- geometric mean inequality). Thus we have the following:
\begin{fact}\label{fact:inradius}
Amongst all the triangles with the same perimeter, the equilateral triangle has the maximum inradius. 
\end{fact}

Putting all these facts together gives our main result.
 
\begin{theorem}
The minimum ribbonlength for the 3-stick folded unknot is $3\sqrt{3}$ where all folds have the same folding information. This occurs when the knot diagram is an equilateral triangle.
\end{theorem}

\begin{proof}
By \thm{local-structure}, we know that the maximum possible width of the  3-stick folded ribbon unknot  where all folds have the same folding information is twice the inradius of the triangle. By Fact~\ref{fact:inradius}, we know that amongst all triangles with the same perimeter, the equilateral triangle has the maximum possible inradius. Thus the minimum possible ribbonlength occurs when the knot diagram is an equilateral triangle. We found the ribbonlength in that case to be $3\sqrt{3}$ in \prop{ooo-oou}.
\end{proof}

%%%%%%%%%%%%%%%%%%%%%%%%%%%%%%%%%%%
%%%%%%%%%%%%%%%%%%%%%%%%%%%%%%%%%%%
\section{Projection stick index and ribbonlength}\label{section:torus knot}

It is natural to wonder about the relationship between the ribbonlength and the number of line segments in the knot diagram. The {\em projection stick number} is the least number of line segments in any projection of a polygonal embedding of a knot. Together with undergraduate students, Colin Adams has given some results about the  projection stick index of knots in~\cite{Adams-Shayler, Adams-PS}. For example, the projection stick index of the trefoil knot is five. 

It turns out that the minimal ribbonlength of a knot (with respect to knot diagram equivalence) does not necessarily occur when the knot diagrams has the projection stick number for that knot. For example, in their paper Kennedy {\it et al.} \cite {KMRT}, found that there was a smaller ribbonlength for the $(5,2)$ and $(7,2)$ torus knots, simply by adding two more sticks to the knot diagram and rearranging. Given their construction, the question remains whether or not the ribbonlength was minimized with respect to link equivalence and not just knot diagram equivalence.

The (standard) projection stick number $(5,2)$ torus knot diagram is found on the right in \fig{T52-3}. The knot diagram of the $(5,2)$ torus knot with smaller ribbonlength (from \cite{KMRT}) is found on the left in \fig{T52-1}. A sequence of Reidemeister moves from one diagram to the other is reasonably straightforward. In \fig{T52-1}, in moving from the left to the middle figure, two RII moves are used to move the edges adjacent to vertex $A$ below edge $EF$. Then one RI and several RIII moves occur between the middle and right diagrams as vertex $B$ is moved down and around vertex $G$.  From \fig{T52-1} right to \fig{T52-3} left, there are two RII and several RIII moves as vertex $B$ is moved down and around vertex $E$. A RI move occurs between \fig{T52-3} left and middle, then the edges are straightened (planar isotopy) to reach \fig{T52-3} right.

\begin{figure}[htbp]
\begin{center}
\begin{overpic}{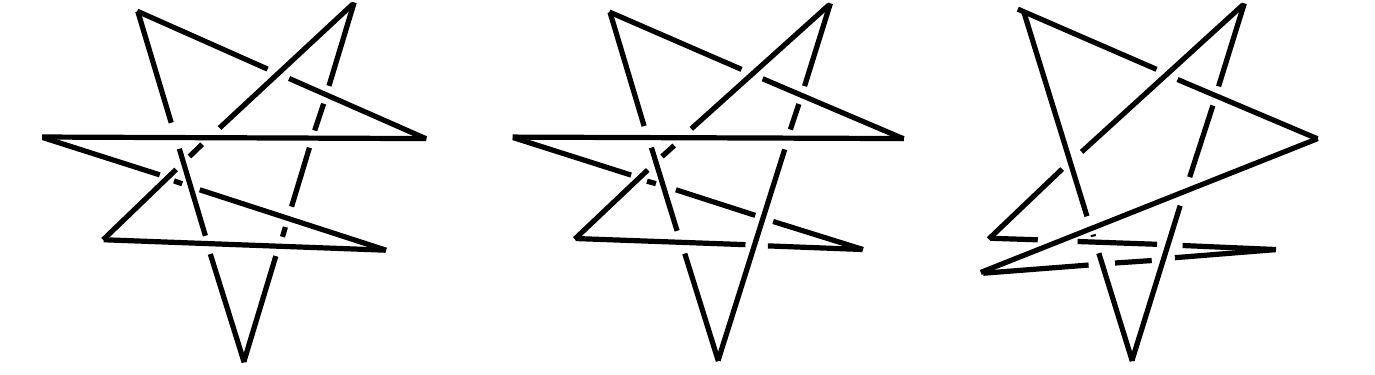}
\put(28,8){$A$}
\put(0,16){$B$}
\put(31.5,16){$C$}
\put(7,26){$D$}
\put(18,0){$E$}
\put(23,27){$F$}
\put(4.5,9){$G$}
\put(63,8){$A$}
\put(34.5,16){$B$}
\put(66,16){$C$}
\put(42,26){$D$}
\put(53,0){$E$}
\put(58,27){$F$}
\put(39,9){$G$}
\put(93,8){$A$}
\put(68,6){$B$}
\put(96,16){$C$}
\put(71,26){$D$}
\put(83,0){$E$}
\put(88,27){$F$}
\put(69,9){$G$}
\end{overpic}
\end{center}
\caption{Reidemeister moves for $(5,2)$ torus knot.}
\label{fig:T52-1}
\end{figure}
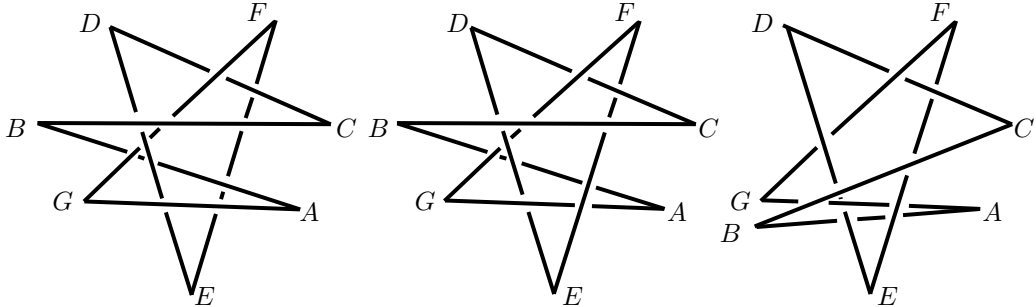

Kaufman in \cite{Kauf90} shows that regular isotopy (the equivalence relation on diagrams generated by the Reidemeister moves of types II and III) corresponds to ambient isotopy of embedded bands, and is in fact more restricted. Thus we expect the ribbon linking numbers the two folded ribbon $(5,2)$ torus knots in question to be different, since we had to use Reidemeister I moves to get from one to the other. Indeed, adding a loop to a diagram is equivalent to adding a full twist to a band. From this observation, it is immediately clear that the two folded ribbon $(5,2)$ torus knots in question are not ribbon link equivalent. We give more details in the following proof.

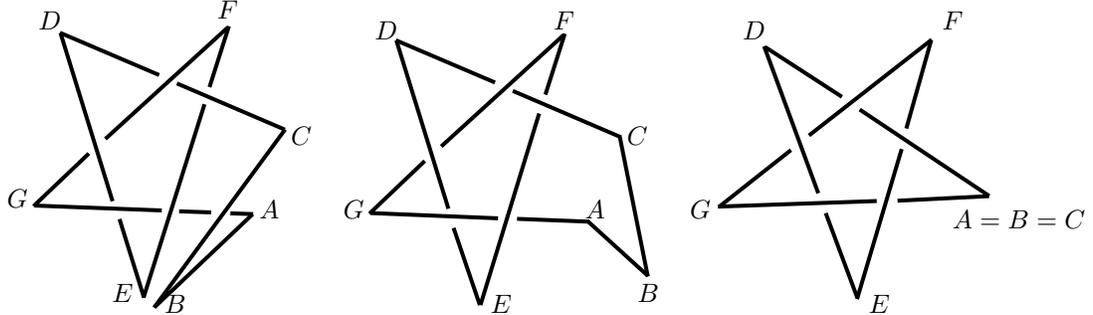
\begin{figure}[htbp]
\begin{center}
\begin{overpic}{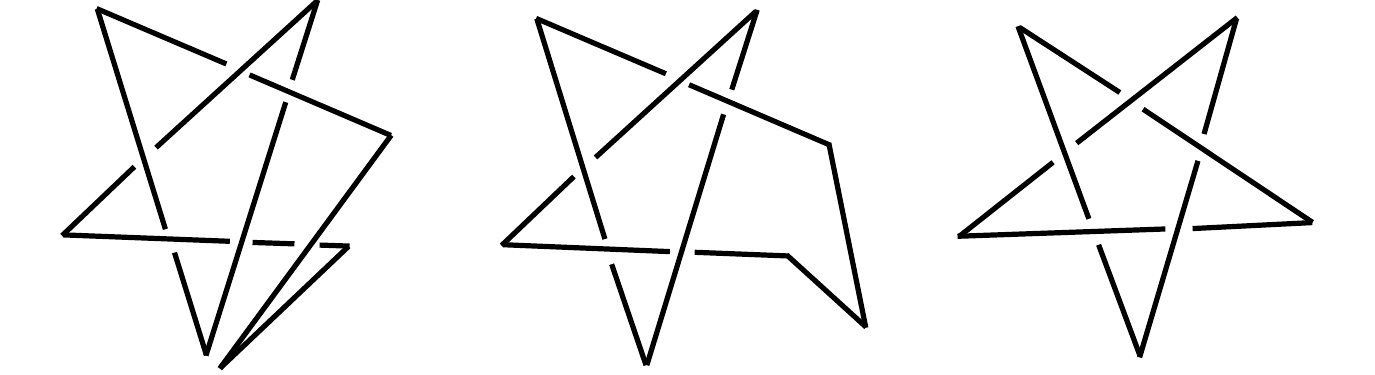}
\put(26,9){$A$}
%\put(4,4){$B$}
\put(17,0){$B$}
\put(29,16){$C$}
\put(5,27){$D$}
%\put(16,0){$E$}
\put(12,1){$E$}
\put(22,28){$F$}
\put(2,10){$G$}
%
%\put(58,8){$A$}
%\put(55,11){$B$}
\put(57,9){$A$}
\put(62,1){$B$}
\put(61,16){$C$}
\put(37,26){$D$}
\put(48,0){$E$}
\put(54,27){$F$}
\put(34,9){$G$}
\put(92,8){$A=B=C$}
\put(72,26){$D$}
\put(84,0){$E$}
\put(91,27){$F$}
\put(67,9){$G$}

\end{overpic}
\end{center}
\caption{Reidemeister moves for $(5,2)$ torus knot.}
\label{fig:T52-3}
\end{figure}

\begin{lemma} The folded $(5,2)$ torus ribbon knot of \fig{T52-1} left is not ribbon link equivalent to the folded $(5,2)$ torus ribbon knot of \fig{T52-3} right.
\end{lemma}

\begin{proof}
We first show that the ribbon linking number is not changed under the Reidemeister moves between \fig{T52-1} left and \fig{T52-3} middle. Then we show that the ribbon knots of \fig{T52-3} middle and right differ by a full twist of the band, and so are not ribbon link equivalent.

It is straightforward to show that RII and RIII moves do not change the ribbon linking number, and we omit these arguments. Instead, we focus our attention on the two RI moves. Without loss of generality, orient the ribbon from $A$ to $B$ to $C$ etc. Observe that between \fig{T52-1} middle and right, edge $AB$ moves below edge $GA$. Since crossing information is continuous, this forces the folding information at a $A$ to be an underfold. Before the RI move, the crossing between $AB$ and $GF$ contributes $-2$, and $A$, as a left underfold, contributes $+1$ to the ribbon linking number. After the RI-move, the crossing is gone and $A$, as a right underfold, contributes $-1$ to the ribbon linking number.  Thus there is no net change in the ribbon linking number.

Similarly, between \fig{T52-3} left and middle, edge $BC$ moves above edge $AB$, forcing the folding information at $B$ to be an overfold. Before the RI-move, the crossing between $GA$ and $BC$ contributes $-2$, and $B$, as a right overfold, contributes $+1$ to the ribbon linking number. After the RI move, the crossing is gone and $B$, as a left overfold, contributes $-1$ to the ribbon linking number. Again, there is no net change in the ribbon linking number.

\fig{T52-4} gives a close-up view of the ribbon near vertices $A$, $B$, and $C$. There is a full twist of the ribbon at vertices $A$ and $B$. This twist can not be removed, unless the folding information is changed at either $A$ or $B$. Thus the folded ribbons knots of \fig{T52-3} middle and right differ by a full twist of the ribbon. Their ribbon linking numbers are different, and they are not ribbon link equivalent.
\end{proof}

\begin{figure}[htbp]
\begin{center}
\begin{overpic}{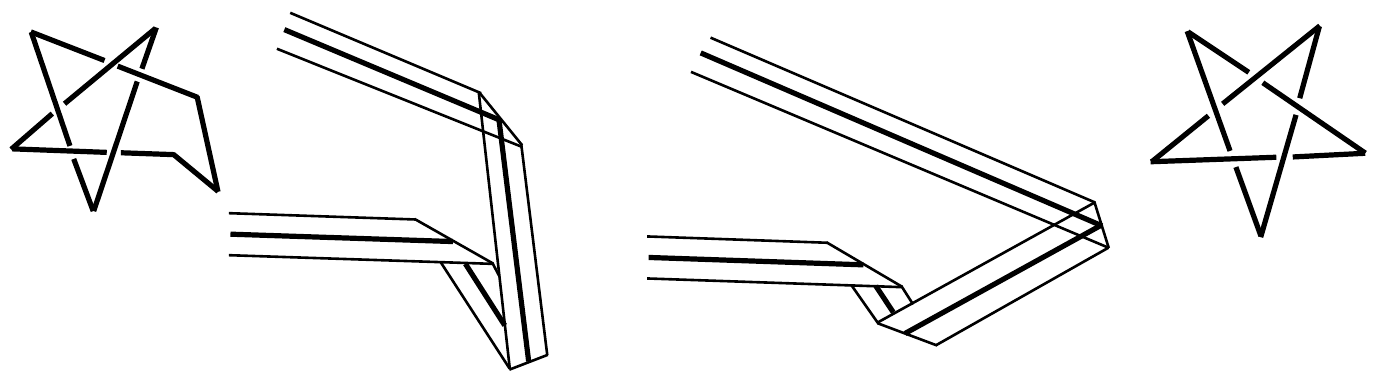}
\put(32.5,11){$A$}
\put(40,0){$B$}
\put(37,19){$C$}
\put(63.5,8.5){$A$}
\put(65,0){$B$}
\put(81,11){$C$}
\end{overpic}
\end{center}
\caption{The ribbon has a full twist concentrated at vertices $A$ and $B$.}
\label{fig:T52-4}
\end{figure}
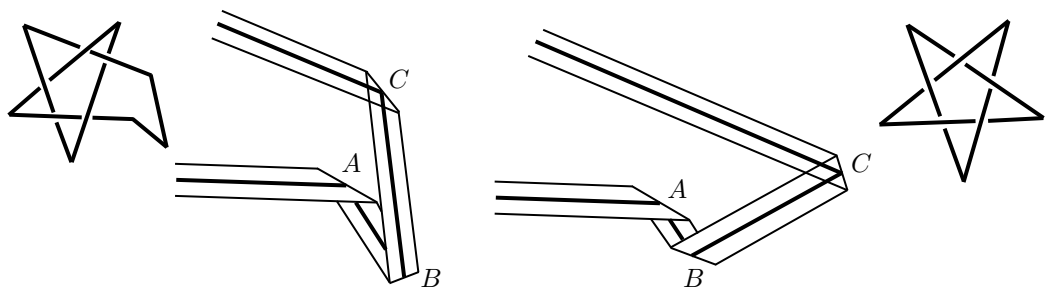

We believe that a similar argument will hold for the standard and shorter versions of the  $(7,2)$ torus knot also found in \cite {KMRT}. 

In summary, these examples give a candidate for the minimal ribbonlength (with respect to knot diagram equivalence) that is given by a knot diagram that is {\bf not} link equivalent to the projection stick number diagram. It thus remains open whether or not the projection stick number diagram provides the minimal ribbonlength among all its link equivalent diagrams.  The precise relationship between projection stick number and ribbonlength remains open. For example, what are the ribbon link numbers generated by a knot diagram with projection stick number? How is this related to ribbonlength?

%%%%%%%%%%%%%%%%%%%%%%%%
%%%%%%%%%%%%%%%%%%%%%%%%

\section{Acknowledgments}
The authors wish to thank John Sullivan and Nancy Wrinkle for discussions on the smooth ribbon case, ribbon width and medial axis, and to Jason Cantarella for discussions on ribbons in general. 

This project couldn't have been completed without the work of Elizabeth Denne's former undergraduate research students from Smith College:  Shivani Ayral, Eleanor Conley, Shorena Kalandarishvili, and Emily Meehan. Thank you.

Finally, we would like to thank the referee for making helpful suggestions on streamlining and clarifying the proofs in Sections 5.2 and 6.

%%%%%%%%%%%%%%%%%%%%%%%%
%%%%%%%%%%%%%%%%%%%%%%%%

\bibliographystyle{amsplain}

\end{document}